\documentclass[a4paper,12pt]{article}
\usepackage[utf8]{inputenc}
\usepackage{color}
\usepackage{microtype}
\usepackage{enumerate}
\usepackage{amsmath}
\usepackage{amssymb}
\usepackage{array}

\usepackage{hyperref}
\usepackage{multirow}
\usepackage[left=2.5cm,top=1.5cm,right=2.5cm, bottom=1.5cm,letterpaper, includeheadfoot]{geometry}
\usepackage{caption}
\usepackage{subcaption}

\newtheorem{lemma}{Lemma}
\newtheorem{theorem}[lemma]{Theorem}
\newtheorem{proposition}[lemma]{Proposition}
\newtheorem{coro}[lemma]{Corollary}
\newtheorem{claim}[lemma]{Claim}

\newtheorem{conjecture}[lemma]{Conjecture}

\newcounter{claim}

\newenvironment{proof}[1][]%
 {\noindent {\setcounter{claim}{0}\sc proof:  }{#1}{}}{\hfill$\Box$\vspace{2ex}} 

	{\noindent {}{#1}{}}{ This proves~(\arabic{claim}).\vspace{1ex}}

\usepackage{tikz}
\tikzstyle{vertex}=[circle, draw, inner sep=0pt, minimum size=6pt]
\newcommand{\vertex}{\node[vertex]}
\usetikzlibrary{arrows,decorations.markings}

\newcommand{\ov}{\overline}
\newcommand{\mc}{\mathcal}
\newcommand{\btw}[3]{d(#1,#3) = d(#1,#2) + d(#2,#3)}

\newenvironment{myp}[1]{{\noindent \it \bf \sf Proof\/ #1:}}{
\hfill {{\vrule height5pt width5pt depth0pt}\hskip 0cm \bigskip}
}

\newcommand{\ol}[1]{\overline{#1}}
\date{}

\title{Lines in bipartite graphs and in 2-metric spaces
\footnote{Partially supported by Basal program AFB170001 and CONICYT Fondecyt/Regular 1180994.}}
\author{M. Matamala $^{a,b}$ and J. Zamora$^{b,c}$ \\
\small ($a$) Depto. Ingeniería Matemática (DIM), Universidad de Chile \\
\small ($b$) Centro de Modelamiento Matemático (CMM, UMI 2807 CNRS), Universidad de Chile\\
\small ($c$) Depto. Matemáticas, Universidad Andres Bello \\}
\begin{document}

\maketitle

\begin{abstract}

The \emph{line generated} by two distinct points, $x$ and $y$, in a finite metric space $M=(V,d)$,
denoted by $\ov{xy}^M$, 
is the set of points given by 
$$\ov{xy}^M:=\{z\in V: d(x,y)=|d(x,z)+d(z,y)| \text{ or } d(x,y)=|d(x,z)-d(z,y)|\}.$$
A 2-set $\{x,y\}$ such that $\ov{xy}^M=V$ is called
a \emph{universal pair} and its generated line a \emph{universal line}.

Chen and Chvátal conjectured that in any finite metric space 
either there is a universal line or there are at least $|V|$ different (non-universal) lines.
Chvátal proved that this is indeed the case when the metric space has distances
in the set $\{0,1,2\}$. 

Aboulker et {\it al.} proposed the following strengthenings for 
Chen and Chvátal conjecture in the context of metric spaces induced by finite graphs: First, the number of lines plus the number of bridges of the graph is at least the number of points. Second, the number of lines plus the number of universal pairs is at least the number of point of the space.

In this work we prove that the first conjecture is true for bipartite graphs different of $C_4$ or $K_{2,3}$, and that the second conjecture is true for metric spaces with distances in the set $\{0,1,2\}$. 
\end{abstract}


\section{Introduction}

In a metric space $M=(V,d)$ a \emph{line} defined by two distinct points $x,y\in V$ is 
the 
subset of $V$ defined by  
$$\ov{xy}^M=\{z\in V: d(x,y)=|d(x,z)+d(z,y)| \text{ or } d(x,y)=|d(x,z)-d(z,y)|\} (\text{see \cite{ChvatalMetric}}).$$

  A line $\ov{xy}^M$ is \emph{universal} if $\ov{xy}^M=V$; in this case $\{x,y\}$ is a \emph{universal pair}.
The number of distinct lines in $M$ is denoted by $\ell(M)$.

In \cite{CC}, Chen and Chvátal proposed the following conjecture.

\begin{conjecture}\label{con:chch}
 Any finite metric space $M=(V,d)$ with at least two points and $\ell(M)<|V|$ 
 has a universal line.
\end{conjecture}

Conjecture \ref{con:chch} is a generalization of a classical result in Euclidean geometry asserting
that every  set of $n$ non-collinear points in the Euclidean plane determines at least $n$ distinct lines (see \cite{dbe}). 

The current best lower bound for the number of lines in a metric space with no universal line is 
$\ell(M)=\Omega(\sqrt n)$ (\cite{metricSpace}).

Although in general the distance function ranges over the non-negative reals,
in order to prove Conjecture $\ref{con:chch}$,  it was observed in \cite{AK} that it is enough to consider non-negative integers. This motivates the definition of  $k$-metric space, with $k$ a positive integer, to be a metric space in which all distances are integral and are at most $k$.
In this context, it was  also proved in \cite{metricSpace} that if $M$ is a $k$-metric space, then the previous bound can be improved to $\ell(M)\geq n/(5k)$, for each $k\geq 3$. 

One particular metric space with integer distances is the metric space induced by a graph. Here the points are the vertices of the graph and the distance 
between two vertices 
is defined by the length of 
a shortest path
between them.
To ease the presentation 
we will refer to 
the metric space induced by a graph $G = (V, E)$, 
just as $G$. Hence, $\ov{xy}^G$  denotes 
the line defined by two distinct vertices $x$ and $y$ in $V$.

In~\cite{BBCCCCFZ} and~\cite{AK} it was proved that Conjecture~\ref{con:chch} holds for metric spaces induced by chordal graphs and for distance-hereditary graphs, respectively. 

The previous results were extended in \cite{AMRZ}, 
where the following stronger result was proved:

\begin{theorem}[Theorem 2.1 in \cite{AMRZ}]\label{th:amrz}
Every graph $G$ such that every induced subgraph of $G$ is either a chordal
graph, has a cut-vertex or a non-trivial module satisfies 
$\ell(G)+\textsc{br}(G)\geq |G|$,
unless $G$ is one of the six graphs depicted in Figure \ref{f:smallgr},
where $\textsc{br}(G)$ is the number of bridges in $G$.
\end{theorem}

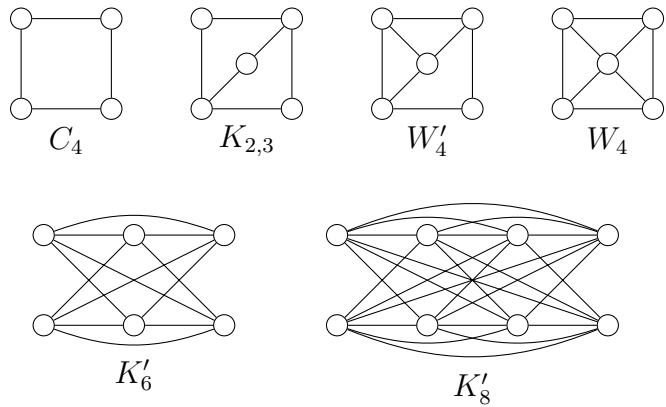
\begin{figure}
\centering
\begin{tikzpicture}[scale=0.6]
\begin{scope}[xshift=0cm]
\vertex[circle,  minimum size=8pt](1) at  (0,0) {};
\vertex[circle,  minimum size=8pt](2) at  (0,2) {};
\vertex[circle,  minimum size=8pt](3) at  (2,2) {};
\vertex[circle,  minimum size=8pt](4) at  (2,0) {};
\draw (1)--(2) -- (3) -- (4)--(1);
\draw (1,-0.7) node{$C_4$}; 
\end{scope}

\begin{scope}[xshift=4cm]
\vertex[circle,  minimum size=8pt](1) at  (0,0) {};
\vertex[circle,  minimum size=8pt](2) at  (0,2) {};
\vertex[circle,  minimum size=8pt](3) at  (2,2) {};
\vertex[circle,  minimum size=8pt](4) at  (2,0) {};
\vertex[circle,  minimum size=8pt](5) at  (1,1) {};
\draw (1)--(2) -- (3) -- (4)--(1);
\draw (1)--(5) -- (3);
\draw (1,-0.7) node{$K_{2,3}$}; 
\end{scope}

\begin{scope}[xshift=8cm]
\vertex[circle,  minimum size=8pt](1) at  (0,0) {};
\vertex[circle,  minimum size=8pt](2) at  (0,2) {};
\vertex[circle,  minimum size=8pt](3) at  (2,2) {};
\vertex[circle,  minimum size=8pt](4) at  (2,0) {};
\vertex[circle,  minimum size=8pt](5) at  (1,1) {};
\draw (1)--(2) -- (3) -- (4)--(1);
\draw (1)--(5) -- (3);
\draw (2)--(5) ;
\draw (1,-0.7) node{$W'_4$}; 
\end{scope}

\begin{scope}[xshift=12cm]
\vertex[circle,  minimum size=8pt](1) at  (0,0) {};
\vertex[circle,  minimum size=8pt](2) at  (0,2) {};
\vertex[circle,  minimum size=8pt](3) at  (2,2) {};
\vertex[circle,  minimum size=8pt](4) at  (2,0) {};
\vertex[circle,  minimum size=8pt](5) at  (1,1) {};
\draw (1)--(2) -- (3) -- (4)--(1);
\draw (1)--(5) -- (3);
\draw (2)--(5) -- (4);
\draw (1,-0.7) node{$W_4$}; 
\end{scope}

\begin{scope}[xshift=0.5cm, yshift=-4.8cm]
\vertex[circle,  minimum size=8pt](1) at  (0,0) {};
\vertex[circle,  minimum size=8pt](2) at  (2,0) {};
\vertex[circle,  minimum size=8pt](3) at  (4,0) {};
\vertex[circle,  minimum size=8pt](4) at  (0,2) {};
\vertex[circle,  minimum size=8pt](5) at  (2,2) {};
\vertex[circle,  minimum size=8pt](6) at  (4,2) {};
\draw (1)--(5) -- (3) -- (4)--(2)--(6)--(1) -- (2) -- (3);
\draw (4)--(5)--(6);
\draw(1)  to[bend right=20]   (3);  
\draw (4)  to[bend left=20]   (6);  
\draw (2,-1.1) node{$K_6'$}; 
\end{scope}

\begin{scope}[xshift=7cm, yshift=-4.8cm]
\vertex[circle,  minimum size=8pt](1) at  (0,0) {};
\vertex[circle,  minimum size=8pt](2) at  (2,0) {};
\vertex[circle,  minimum size=8pt](3) at  (4,0) {};
\vertex[circle,  minimum size=8pt](4) at  (6,0) {};
\vertex[circle,  minimum size=8pt](5) at  (0,2) {};
\vertex[circle,  minimum size=8pt](6) at  (2,2) {};
\vertex[circle,  minimum size=8pt](7) at  (4,2) {};
\vertex[circle,  minimum size=8pt](8) at  (6,2) {};
\draw (1)--(6)--(3) --(8) --(2) --(5) --(4) --(7) --(1) --(8)  ;
\draw (5)--(3);
\draw (4)--(6);
\draw (2)--(7);
\draw (1)--(2)--(3)--(4);
\draw (5)--(6)--(7)--(8);
\draw(1)  to[bend right=18]   (3);  
\draw (5)  to[bend left=18]   (7);  
\draw(2)  to[bend right=18]   (4);  
\draw (6)  to[bend left=18]   (8);  
\draw(1)  to[bend right=22]   (4);  
\draw (5)  to[bend left=22]   (8);
\draw (3,-1.4) node{$K_8'$}; 
\end{scope}
\end{tikzpicture}
  \caption{Graphs excluded in Theorem \ref{th:amrz}.}
  \label{f:smallgr}
\end{figure}

Given this result, the authors in \cite{AMRZ} proposed the following conjecture:

\begin{conjecture}[Conjecture 2.2 in \cite{AMRZ}]\label{c:ampzBR}
There is a finite set of graphs ${\cal F}_0$ such that every connected graph $G \notin \mathcal F_0$ either has a pendant edge or satisfies  $\ell(G)+\textsc{br}(G)\geq |G|$.
\end{conjecture}

In this work, we prove that a bipartite graph $G$ satisfies $\ell(G)+\textsc{br}(G)\geq |G|$ unless $G \in \{ C_4, K_{2,3}\}$. 
The proof is based on the study of the lines defined by vertices at distance 2. In this context, we prove two interesting results: first, we prove 
that given two vertices $x$ and $y$ at distance two in a graph $G$, the graph induced by $\overline{xy}^G$  either has diameter two or has a cut vertex in $\{x,y\}$. As a consequence, a 2-connected graph $G$ of diameter at least three 
can not have a universal pair whose vertices are at distance two.

Second, we prove that 2-connected bipartite graphs have more lines than vertices. We do that counting the lines generated by vertices at distance 2. At first glance, this restriction made the problem harder as it reduces the number of pairs of vertices that can generate lines. 
However, it also reduces the possibilities for two pairs of vertices to generate the same line. We think that this trade-off can be exploited in other contexts as well, since in general, it is not easy to characterize pairs of vertices that define the same line.

Our result also proves, for bipartite graphs, the following conjecture made by Zwols \cite{YZ}: if $\ell(G)<|G|$, then either $G$ has a bridge or it contains $C_4$ as induced subgraph. It also allows to extend Theorem \ref{th:amrz}, by adding bipartite graphs as an option for the induced subgraphs.

Notice that graphs of Figure \ref{f:smallgr} satisfy Conjecture \ref{con:chch} because they have  universal lines. Moreover, they have more than one pair of vertices that define universal lines. This is a phenomena that appears in all the examples of graphs with few different lines. Inspired in this observation, the following conjecture was proposed in \cite{AMRZ}:

\begin{conjecture}[Conjecture 2.3 in \cite{AMRZ}]\label{c:ampzUP}
Let $G=(V,E)$ be a connected graph with at least two vertices.
Then, $\ell(G)+\textsc{up}(G)\geq |V|$, where $\textsc{up}(G)$ denotes the number of universal pairs in $G$.
\end{conjecture}

In this work we study this conjecture in a more general setting. In particular, we prove that each 2-metric space $M=(V,d)$ satisfies
$\ell^*(M ) + \textsc{up}(M)\geq |V|$,
where
$\ell^*(M)$ denotes the number of distinct non-universal lines in $M$
and $\textsc{up}(M)$ denotes the number of universal pairs in $M$.
Notice that when $\textsc{up}(M ) = 0$ we have that $\ell(M ) = \ell^*(M )$.
Hence, our result implies that
Conjecture \ref{con:chch} holds for 2-metric spaces,
a result previously proved in \cite{ChCh11,Chvatal2}.
Our proof is from first principles then giving an alternative proof for this fact.

An important role in this work is played by \emph{pair of twins}.
We say that $(v,v')$ is a \emph{pair of twins} of a metric space $M=(V,d)$,
if $v$ and $v'$ are two distinct points in $V$ such that $d(v,v')\neq 1$ and for all $u\notin \{v,v'\}$, $d(v,u)=d(u,v')$. In a metric space induced by a connected graph, a pair of twins is usually called a pair of \emph{false} twins.

\section{Metric spaces defined by finite graphs}

In a metric space induced by a graph $G$, the distance between two vertices
is the length of a shortest path between them. 
As usual, $N_G(x)$ will denote the neighborhood of the vertex $x$.


Although our main result is about metric spaces defined by bipartite
graphs, we start by proving some preliminaries results which are valid for arbitrary graphs.
We shall use them in the proof of our main result.



A crucial point in our development is that we only count lines defined by
vertices at distance two. The following lemma shows part of the structure of these lines: 

\begin{lemma}\label{t:candy} 
 Let $x,y$ be vertices of $G$ at distance 2. If two vertices $a$ and $b$ are such that $d(a,x)=d(a,y)+d(y,x)$ and $d(b,y)=d(b,x)+d(x,y)$, 
 then any path $P$ between $a$ and $b$ contained in $\ov{xy}^G$ contains the set $\{x,c',y\}$,  for some  $c'\in N_G(x)\cap N_G(y)$.
\end{lemma}

\begin{myp}{}
For each $v \in G$, we define the function $\Delta(v) := d(y,v) - d(x,v)$. 
Since $d(x,y)=2$, the function $\Delta$ takes only values in $\{-2, -1, 0, 1, 2\}$; moreover, for every $u \in \ov{xy}^G, \Delta(u) \in \{-2,0,2\}$.  

Since $a  \in \ov{xy}^G$ and $d(a,x)=d(a,y)+d(y,x)$, then $\Delta(a) = -2$. Equivalently,  we deduce $\Delta(b) = 2$.  
Notice that for two adjacent vertices $u$ and $v$ we have that $|\Delta(u) -\Delta(v)| \leq 2$; hence, 
for $u$ and $v$ adjacent and both in $\ov{xy}^G$, we have that $|\Delta(u) -\Delta(v)| \in \{0,2\}$. We deduce that there must exist a vertex $c'$ in $P$ such that $\Delta(c') = 0$. Let us assume that $c'$ is the first vertex in $P$ from $a$ to $b$ such that $\Delta(c')=0$. Since $c' \in \ov{xy}^G$, then $c'\in N_G(x)\cap N_G(y)$ and
the neighbor $w$ of $c'$ in $P$ closer to $a$ satisfies $\Delta(w)=-2$ and $d(x,w) \leq 2$; it follows that $d(w,y)=0$, which implies that $w=y\in P$. With a similar argument applied to $b$ we can prove that $x\in P$.
\end{myp}

\begin{coro}\label{c:candy}
Let $x, y$ be two vertices of $G$ at distance $2$. Let $z \in \ov{xy}^G$ with $z \notin N_G(x) \cap N_G(y)$ and let $P$ be a path between $z$ and $x$ such that 
$P \subseteq \ov{xy}^G$ and $y \notin P$. Then $d(z,y)=d(z,x)+d(x,y)$.
\end{coro}

\begin{myp}{}
Since $z \notin N_G(x) \cap N_g(y)$, then $d(x,y) \neq d(x,z) + d(z,y)$.
By contradiction suppose that $d(z,x)=d(z,y)+d(x,y)$. Since $P \subseteq \ov{xy}^G$, Lemma \ref{t:candy} implies  $y\in P$, which is a contradiction. 
\end{myp}

\begin{coro}\label{c:universal}
Let $G=(V,E)$ be a 2-connected graph and let $x, y$ be two vertices of $G$ at distance $2$. If $\ov{xy}^G$ is a universal line, then $(x,y)$ is a pair of twins and $V=\{x,y\}\cup N_G(x)$.
\end{coro}

\begin{myp}{}
By contradiction suppose there exists a vertex $z$ which is neighbor of $x$ but not of $y$; $z \in \ov{xy}^G$ because $(x,y)$ is a universal pair. By Corollary \ref{c:candy} we have $d(z,y)=d(z,x)+d(x,y)$.  Moreover, every path between $z$ and $y$ contains $x$, by Lemma \ref{t:candy}. This implies that $x$ is a cut vertex; a contradiction because $G$ is a 2-connected graph. 
\end{myp}

Corollary \ref{c:universal} implies that lines defined by vertices at distance 2 are non universal in  2-connected graphs with diameter at least three. 
This motivates us to count the number of distinct lines defined by vertices at distance two. 
The set of lines defined by vertices at distance two is denoted by ${\cal L}_2^G$ and its cardinality by $\ell_2(G)$. For a subset $U$ of vertices of $G$ we shall denote ${\cal L}_2^G(U)$ the set of lines defined in $G$ by two vertices in $U$ at distance two.

The next lemma is a refinement of part $(2)$ in the proof of Theorem 2.1 in \cite{AMRZ}. 
Here, instead of considering arbitrary lines, we only consider lines  defined by vertices at distance two.
The proof is the same, but we present it here for the sake of completeness.

\begin{lemma}\label{lem:cutvertex} Let $G$ be a bridgeless graph
such that $G =G_1\cup G_2$, $V(G_1)\cap V(G_2)=\{v\}$ and $E(G)=E(G_1)\cup E(G_2)$. 
Then, 
$$\ell_2(G)\geq \ell_2(G_1)+\ell_2(G_2)-1+|N_{G_1}(v)||N_{G_2}(v)|.$$
\end{lemma}
\begin{myp}{} Let $V_i=V(G_i)$, for $i=1,2$.
It is easy to see that for each pair $u,v\in V_i$ we have that 
\begin{equation}\label{eq:cutset}
\ov{uv}^G\in\{\ov{uv}^{G_i},\ov{uv}^{G_i}\cup V_{3-i}\} \text{, for } i=1,2.
\end{equation}
Therefore, 
at most one line belongs to the intersection ${\cal L}_2^G(V_1)\cap {\cal L}_2^G(V_2)$; hence, there are at least $\ell_2(G_1)+\ell_2(G_2)-1$ lines  in ${\cal L}_2^G(V_1)\cup {\cal L}_2^G(V_2)$.

Now we prove that there are at least $|N_{G_1}(v)||N_{G_2}(v)|$ lines of $G$ not in ${\cal L}_2^G(V_1)\cup {\cal L}_2^G(V_2)$. Let $u_i$ be a neighbor of $v$ in $G_i$, for each $i=1,2$. We have that $\ol{u_1u_2}^G\in \mc L_2^G$ and it contains exactly one neighbor $u_i$ of $v$ in $G_i$, for each $i=1,2$. Since $v$ has degree at least two in $G_i$, for each $i=1,2$, as otherwise $G$ has a bridge, at least one neighbor of $v$ in $G_i$ does not belong to $\ol{u_1u_2}^G$; it follows from (\ref{eq:cutset}) that $\ol{u_1u_2}^G\notin {\cal L}_2^G(V_1)\cup {\cal L}_2^G(V_2)$.

Let $u_i,v_i$ be neighbors of $v$ in $G_i$, for each $i=1,2$. We
have that $\{u_1,u_2\}\neq \{v_1,v_2\}$ implies $\ol{u_1u_2}^G\neq \ol{v_1v_2}^G$; 
then, there are at least $ |N_{G_1}(v)||N_{G_2}(v)|$ lines in ${\cal L}_2^G\setminus ({\cal L}_2^G(V_1)\cup {\cal L}_2^G(V_2))$.

\end{myp}

\subsection{Bipartite graphs}

In this section we consider metric spaces defined by bipartite graphs. 

Our starting point is the following simple observation: given a vertex $v$ in a bipartite graph and two vertices $u$ and $w$ in $N_G(v)$, we have that $N_G(v)\cap \ov{uw}^G=\{u,w\}$; it follows that for each vertex $v$ in a bipartite graph $G$, $\ell_2(G)\geq \binom{d(v)}{2}$. Hence, \emph{locally}, a vertex in a bipartite graph has many pairs of vertices that defines different lines. 

Two problems appear when one tries to move this idea from local to global. On the one hand, two o more vertices can have the the same neighborhoods (pairs of twins or modules);  on the other hand, the same line can be generated by different pairs in several neighborhoods.  

Both problems appear in $C_4$, where $\ell_2(C_4)=\ell(C_4)=1$. This graph has two pairs of twins and every pair of vertices at distance two generates a universal line.

The first situation also appears in $K_{2,3}$, where $\ell_2(K_{2,3})=\ell(K_{2,3})=4$. 
In  this case, all the vertices in the bigger independent set have the same neighborhood.
In the following figure we show two cases where the second problem appears: 
\begin{figure}[h]\label{f:sameline}
\centering
\begin{minipage}[c]{0.5\textwidth}
\centering
\begin{tikzpicture}

\vertex (A)  {$y$};
\vertex (B) [below right of=A] {};
\vertex (C) [above right of=A] {};
\vertex (D) [above right of=B] {$x$};
\vertex (E) [above right of=D] {};
\vertex (F) [below right of=D] {};
\vertex (G) [above right of=F] {$t$};

\draw (A) -- (C) -- (D) -- (E) -- (G) -- (F) -- (D) -- (B) -- (A);

\end{tikzpicture}
\subcaption{}
\end{minipage}%
\begin{minipage}[c]{0.5\textwidth}
\centering
\begin{tikzpicture}

\vertex (A)  {$y$};
\vertex (B) [below right of=A] {};
\vertex (C) [above right of=A] {};
\vertex (D) [above right of=B] {$x$};
\vertex (H) [above right of=C] {};
\vertex (I) [right of=D] {};
\vertex (J) [right of=H] {};
\vertex (K) [right of=J] {};
\vertex (L) [right of=I] {$s$};
\vertex (E) [above right of=L] {};
\vertex (F) [below right of=L] {};
\vertex (G) [above right of=F] {$t$};

\draw (D) -- (B) -- (A) -- (C) -- (H) -- (J) -- (K) -- (E) -- (G) -- (F) -- (L);
\draw (C) -- (D) -- (I) -- (L) -- (E);

\end{tikzpicture}

\subcaption{}
\end{minipage}
\caption{In (a) $\ov{yx} = \ov{xt}$ and in (b) $\ov{yx} = \ov{st}$ }
\end{figure}
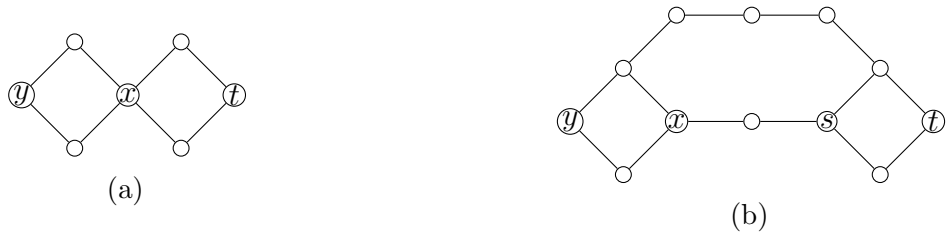
\bigskip

However, the following lemma shows that the existence of many lines locally is enough to satisfy Conjecture \ref{c:ampzBR} for complete bipartite graphs.
\begin{lemma}\label{l:complete}
 If $G=K_{p,q}$ with $2\leq p\leq q$, then $\ell_2(G)= \binom{p}{2}+\binom{q}{2}$ unless $p=q=2$.
 In particular, if $p+q\geq 6$, then $\ell_2(G)\geq p+q=|G|$.
\end{lemma}
\begin{myp}{} 
Let $X$ and $Y$ be the independent sets of $K_{p,q}$. Given two vertices $a$ and $b$ of $G$ at distance two we have  $\ov{ab}^G=X\cup \{a,b\}$, if $a,b\in Y$ and $\ov{ab}^G=Y\cup \{a,b\}$, if $a,b\in X$.
Hence, when $p+q\geq 5$ each pair of vertices in the same independent set defines a distinct line in ${\cal L}_2^G$.
\end{myp}

In order to control the second problem, we need to characterize the pair of vertices that define the same line.
We define the \emph{width} of a line $\ell\in {\cal L}_2^G$ as the number of pair of  vertices $\{x,y\}$ with $d(x,y)=2$ and $\ell=\ov{xy}^G$. 
We now prove that the existence of lines of width at least two forces some structure of the graph. We use this structure to prove, in the next section, our main result.

Let $N^2_G(y)$ denotes the set of vertices at distance two of $y$. Given four vertices
$y,x,s$ and $t$, we say that $[yxst]$ holds if there is a shortest path $P$ between $y$ and $t$ containing $x$ and $s$ such that $x$ belongs to the subpath of $P$ between $y$ and $s$.
Equivalently, $[yxst]$ holds if and only if 
\begin{eqnarray*}
d(y,t)&=&d(y,x)+d(x,t)\\
& = & d(y,x)+d(x,s)+d(s,t)\\
& = & d(y,s)+d(s,t).
\end{eqnarray*}

To ease the presentation we denote by $xPy$ the subpath of a path $P$ 
between two of its vertices $x$ and $y$.

\begin{proposition}\label{p:uniqueline} Let $G$ be a bipartite graph, 
 $x,y,s$ and $t$ vertices of $G$ such that $d(x,y)=d(s,t)=2$ and $\ov{xy}^G=\ov{st}^G$. If $[yxst]$ holds, then either $y$  (resp. $t$) is a cut vertex,
 or it is dominated by $x$ (resp. $s$). 
 
 Moreover, when $G$ is 2-connected we have the following:
 \begin{enumerate}[(i)]
  \item \label{s:zneighbor}  
  For each $z\in N^2_G(y)$ and each $w\in N^2_G(t)$, 
  $d(z,t)=d(x,t)=d(y,s)=d(w,y)$.
\item  \label{s:ztot} For each $z\in N^2_G(y), z\neq x$ and for each $v$ for which $d(z,t)=d(z,v)+d(v,t)$ holds, $d(v,s)=d(v,t)$.
Similarly, for each $w\in N^2_G(t),w\neq s$ and for each $u$ for which $d(y,w)=d(y,u)+d(u,w)$ holds,
$d(u,y)=d(u,x)$.

\item \label{s:cycle} The vertices $x$ and $s$ belongs to an induced cycle of length $2(d(x,s)+2)$.
\end{enumerate}
\end{proposition}

\begin{myp}{} To prove the first part, we proceed by contradiction assuming that $y$ is neither a cut vertex nor dominated by $x$. 

Since $y$ is not dominated by $x$, there exists $b\in N_G(y) \setminus N_G(x)$. From the  definition of $\ov{xy}^G$, we have that $b\in \ov{xy}^G$ and $d(b,x)=d(b,y)+d(y,x)$. 

Let $P$ be a path from $b$ to $t$ not containing $y$. It exists as $y$ is not a cut vertex.  Since $d(b,x)=d(b,y)+ d(y,x)$, $d(y,t)=d(y,x)+d(x,t)$ and $y$ does not belong to $P$; from Lemma \ref{t:candy} we deduce that $P$ is not completely contained in $\ov{yx}^G$. 

Let $w' \in P$ be the closer vertex to $b$ which is not in $\ov{yx}^G$, and $w$ be its neighbor in $P$ closer to $b$, which will belong to $\ov{xy}^G$.
 
If $w\in N_G(x)\cap N_G(y)$, then the path $bPwx$ would be completely contained in $\ov{xy}^G$; but this contradicts Lemma \ref{t:candy} because $d(b,x)=d(b,y)+d(y,x)$ and $d(y,x)=d(y,x)+d(x,x)$. 

Since $w'$ is the first vertex not in $\ov{yx}^G$ of $P$, the path $ybPw \subseteq \ov{yx}^G$. From Lemma \ref{t:candy} we get that this path does not contain $x$; it follows from Corollary \ref{c:candy} that 
\begin{equation} \label{eq:1}
d(w,x)=d(w,y)+d(y,x) = d(w,y) + 2.
\end{equation}

Since the graph $G$ is bipartite and $w' \notin \ov{xy}^G$, we have that $d(x,w')=d(y,w')$. As $|d(y,w)-d(y,w')| \leq 1$ and  $|d(x,w)-d(x,w')|\leq 1$, it follows  that the only way to satisfy Equation \ref{eq:1} is when  
\begin{equation} \label{eq:2}
d(y,w')=d(y,w)+1. 
\end{equation}

Since $[yxst]$ holds, there exist a $y$-$s$-path $Q$ contained in $\ov{st}^G$ that does not contain $t$. Hence, the path $wPbQ$ is a $w$-$s$-path contained in $\ov{st}^G$ and Corollary \ref{c:candy} implies $\btw{w}{s}{t}$. As before, we conclude that
\begin{equation} \label{eq:3}
d(s,w')=d(s,w)+1.
\end{equation}

Since $\btw{w}{s}{t}$, we get that a shortest $w$-$s$-path must be contained in $\ov{st}^G = \ov{yx}^G$, which implies it contains the vertices $y$ and $x$ (Lemma \ref{t:candy}); in particular, we have that  

\begin{equation} \label{eq:4}
d(w,s) = d(w,y) + d(y,s).
\end{equation}

The following chain of equality holds:
\begin{align*}
d(s,w')& = d(s,w)+1 \tag*{(by (\ref{eq:3}))} \\ 
       & = d(s,y)+d(y,w)+1   \tag*{(by (\ref{eq:4}) )} \\
       & = d(s,y)+d(y,w')  \tag*{(by (\ref{eq:2}))}\\
       & = d(s,x)+d(x,y)+d(y,w') \\ 
       & = d(s,x) + d(x,y) + d(x,w') \tag*{($w' \notin \ov{yx}^G$)} \\
       & \geq d(s,w') + d(x,y)  
\end{align*}
which implies $d(x,y)= 0$, a contradiction. Hence, there is not such vertex $b$ and the vertex $y$ is a cut vertex or it is dominated by $x$.

Now we assume that $G$ is a 2-connected graph. 
\begin{enumerate}[(i)]

\item We first prove that for each $z\in N^2_G(y)$, $d(z,t)=d(x,t)$. It is obvious for $z=x$. For $z \neq x$ we have that $d(x,t)+2=d(y,t)\leq d(z,t)+2$, by the triangle inequality. Hence, $d(x,t)\leq d(z,t)$. 

The other inequality comes from the fact that $x$ dominates $y$; which implies that $d(z,t)=d(z,s)= 2$; hence, $d(z,t)=d(z,s)\leq d(x,s)+2=d(x,t)$. 

By a symmetric argument, for each $w\in N^2_G(t)$ 
we have that $d(w,y)=d(s,y)$.  As $d(y,t)=d(y,s)+2=d(x,t)+2$ we get the result. 

\item Let $v$ be such that $d(z,t)=d(z,v)+d(v,t)$ holds. On one hand, since $z \notin \ov{st}^G$, we have that  
$$ d(z,v)+d(v,t) = d(z,t) = d(z,s) \leq d(z,v)+d(v,s) $$

which implies that $d(v,t) \leq d(v,s)$; on the other hand, since $s$ dominates $t$, we have that $d(v,s)\leq d(v,t)$.Hence, $d(v,s)=d(v,t)$. The symmetric analysis shows the statement for each $u$ satisfying $d(y,w)=d(y,u)+d(u,w)$.

\item Let $z\in N^2_G(y)$ and let $P$ be a shortest path between $z$ and $t$. We denote by $w = P\cap N^2_G(t)$ and by $u = N_G(t)\cap P$. Notice that by (\ref{s:ztot}), no vertex in $zPw$ belongs to $\ov{st}^G$.

Let $Q$ be a shortest path between $x$ and $s$. Then, 
$Q$ and $P$ are vertex disjoint, because $Q$ is contained in $\ov{st}^G$,
which implies that $C=vPuQ$ is a cycle containing $x$ and $s$, where $v\in N_G(y)\cap N_G(z)\cap N_G(x)$. 

Now we prove that the cycle is induced. Assume that there is a chord $ab$ in $C$. If $a=v$ then $b \in N^2(y)$ and contradicts the fact that $Q$ and $P$ are shortest paths. A similar analysis shows that $u$ can not be a vertex of the chord. Hence, we can assume that $a\in P$ and $b\in Q$. From triangular inequality we get that  
$$d(z,t)=d(z,s)\leq d(z,a)+1+d(b,s) \text{ and }  d(w,y)=d(w,x)\leq d(w,a)+1+d(b,x); $$
but, we know from (\ref{s:ztot}) that $d(z,t) = d(x,s)+2 = d(w,y)$. Replacing in the previous inequalities and summing them we obtain
$$ 2d(x,s) + 4 \leq d(z,a)+d(a,w)+d(x,b) + d(b,s) + 2.$$

Since $P$ and $Q$ are shortest path, it follows that $\btw{z}{a}{w}$ and $\btw{x}{b}{s}$ which imply
$$ d(x,s) + 2 \leq d(z,w)$$

a contradiction because $d(x,s) = d(z,w)$ by (\ref{s:ztot}).
\end{enumerate}
\end{myp}

In order to apply Proposition \ref{p:uniqueline} we need to understand in which situations 
two pairs of vertices $x,y$, and $s,t$, with $d(x,y)=d(s,t)=2$, and
generating the same line, do satisfy $[yxst]$. 

\begin{lemma} \label{l:whenaligned}
Let $y,x,s,t$ such that $\{x,y\}\neq \{s,t\}$, $d(x,y)=d(s,t)=2$ and $\overline{xy}^G=\overline{st}^G$.
If $G$ is 2-connected, bipartite and has no pairs of twins, then $\max\{d(x,s),d(x,t),d(y,s),d(y,t)\}$
is at least four. Moreover, if $d(y,t)=\max\{d(x,s),d(x,t),d(y,s),d(y,t)\}$, then $[yxst]$ holds. 
\end{lemma}
\begin{proof}
Let $\beta=\max\{d(x,s),d(x,t),d(y,s),d(y,t)\}$ and let $y$ and $t$ be such that $\beta=d(y,t)$.
Since $\{x,y\}\neq \{s,t\}$ we have that $\beta\geq 1$.
If $\beta=1$, then $d(x,s)=d(x,t)=d(y,s)=d(y,t)=1$.
Moreover, since $(x,y)$ is not a pair of twins, there is $z$ which is adjacent to $y$ 
and not adjacent to $x$; then, $z\in \overline{xy}^G$ and, since $\beta=1$, $z\notin \{s,t\}$.
As $G$ is bipartite, $d(z,s)=d(z,t)=2$ which implies the contradiction $z\notin \overline{st}^G$,
since $\overline{xy}^G=\overline{st}^G$.

When $\beta=2$ we cannot have $d(y,s)=d(y,t)+2$. As  $y\in \overline{st}^G$
we get that $d(y,t)=d(y,s)+2$ which implies that $s=y$. 
Similarly, as $t\in \overline{xy}^G$ we conclude that $x=t$. Thus, we get the contradiction
$\{x,y\}=\{s,t\}$.

As before, when $\beta=3$ we cannot have $d(y,s)=d(y,t)+2$,  hence,  
$d(y,t)=d(y,s)+2$ and then $d(y,s)=1$.
Similarly, as $t\in \overline{xy}^G$ we conclude that $d(x,t)=1$.

Let $z\in N_G(x)\cap N_G(y)$ which imply $z\in \overline{xy}^G$.
As $d(x,t)=1$ and $d(y,s)=1$ we have that $d(z,s),d(z,t)\leq 2$; but 
$d(y,t)=3$ implies $d(z,t)=2$. Since $z\in \overline{st}^G$ we get that $z=s$.
In a similar way we can prove that $N_G(s)\cap N_G(t)=\{x\}$.

By Corollary \ref{c:universal}, in a 2-connected graph with no pairs of twins 
there is no universal pairs.
We shall get a contradiction by proving that $(x,y)$ is a universal pair.
Let us assume that $z\notin \ov{xy}^G=\ov{st}^G$.
Then, $z\notin \{x,y,s,t\}$. 
As $\beta=3$ a shortest path $P$ between $z$ and $x$ either contains $y$ or contains $t$. 
In the first situation, $z\in \overline{xy}^G$, so we can assume
that $t$ is in $P$, that is to say, $d(z,x)=d(z,t)+1$.
By a symmetric argument we can assume that a shortest path between $z$ and $s$ must contains
$y$. Hence, $d(z,s)=d(z,y)+1$.

Since $z\notin \ov{xy}^G$ and $G$ is bipartite we know that $d(z,x)=d(z,y)\geq 2$.
Therefore, $$d(z,s)=d(z,y)+1=d(z,x)+1=d(z,t)+2$$ which contradicts $z\notin \ov{st}^G$.

Therefore $\beta=d(y,t)\geq 4$.
Since $t\in \ov{xy}^G$ and $y\in \ov{st}^G$ we get that
$d(y,t)=d(x,t)+d(x,y)$ and $d(y,t)=d(y,s)+d(s,t)$. 

For $z\in N_G(x)\cap N_G(y)$ we have that 
$d(y,t)\leq d(z,t)+1\leq d(x,t)+2=d(y,t)$ and then 
$$d(z,t)=d(x,t)+1=d(y,t)-1\geq 3.$$
We also have that $d(z,s)\leq d(y,s)+1=d(y,t)-1=d(z,t)$.
Since $z\in \ov{st}^G$ we get that $d(z,t)=d(z,s)+2$.
By using this equality we get that 
$$d(x,s)\leq d(z,s)+1=d(z,t)-1=d(x,t).$$
Since $x\in \ov{st}^G$ we get that $d(x,s)+2=d(x,t)$
and then $[yxst]$ holds.

\end{proof}

\subsection{Proof of the main result}


In this section we prove our main result. We start by considering 
2-connected graphs without pairs of twins.

\subsubsection*{2-connected bipartite graphs with no pairs of twins}

Before proving our result we need some definitions. Let $G$ be a 2-connected graph with no pairs of  twins and let $x,y$ be vertices of $G$ such that $w(\ov{xy}^G)>1$. From Lemma \ref{l:whenaligned}
and Proposition \ref{p:uniqueline} we know that $x$ dominates $y$ or $y$ dominates $x$, since none of them is a cut vertex. As $G$ has no pairs of twins only one of these options can hold. We define $X$ as the set of vertices $x$ such that there is a vertex $y\in N^2_G(x)$ with $w(\ov{xy}^G)>1$ and such that $x$ dominates $y$. 

For each $x\in X$, let $Y_x$ be the set of vertices $y\in N^2_G(x)$ with 
$w(\ov{xy}^G)>1$ and set $Y=\cup_{x\in X}Y_x$.

\begin{lemma}\label{l:xcapy}
For $X$ and $Y$ defined above, $X \cap Y = \emptyset$ when $G$ is a 2-connected bipartite
graph without pairs of twins.
\end{lemma}

\begin{myp}{}
By contradiction, suppose there exists $y \in X \cap Y$. 
As $y\in Y$, there is $x \in X$ such that $y \in Y_x$. Let $s,t \in V$ such that $\ov{yx}^G = \ov{st}^G$. 
Since $G$ is a 2-connected bipartite
graph without pairs of twins, by Lemma \ref{l:whenaligned} we have that 
$d(y,t)=\max\{d(x,s),d(x,t),d(y,s),d(y,t)\}\geq 4$ and $[yxst]$ holds.
By part (\ref{s:zneighbor}) of  Proposition \ref{p:uniqueline} we know that $x$ dominates $y$.

Since $y\in X$, there is $z \in N^2(y)$ such that $y$ dominates $z$. 
From Proposition \ref{p:uniqueline}  we know that $d(z,t) = d(x,t)=d(y,t)-2$.
But then we get the contradiction: $d(y,t)\leq d(z,t)< d(y,t)$.   
\end{myp}


From Corollary \ref{c:universal} we know that a 2-connected graph $G$ without pairs of twins has no universal pairs $(x,y)$, with $d(x,y)=2$. 
Hence, in order to prove our result for these graphs, we have
to prove that there are at least $|G|$ distinct non-universal lines defined by pairs 
of vertices at distance two.

To this end, we define a function $f$ from the set of vertices of the graph into the set of lines of $G$. The function $f$ associates to each vertex $v$ a line generated by $v$ and a vertex in $N^2_G(v)$, that we denote by $g(v)$. 
If $f$ is injective, then the number of distinct non-universal lines defined by pairs 
of vertices at distance two is al least the number of vertices, and we are done.

Function $f$ could not be injective for two reasons.
The first reason is that there are distinct vertices $u$ and $v$ such that 
$\{v,g(v)\}=\{u,g(u)\}$. This is equivalent to $g^2(u)=u$.

If $g^2(w)=w$ for no vertex $w\in V$, 
then $f$ still could not be injective
if there are distinct vertices $u$ and $v$ 
such that $\overline{vg(v)}^G=\overline{ug(u)}^G$.
From Proposition \ref{p:uniqueline}
we know that in this case $v\in X$ and $g(v)\in Y$ or $v\in Y$ and $g(v)\in X$. 
Hence, either $v$ dominates $g(v)$ or $g(v)$ dominates $v$. 

Therefore, when defining $g(v)$ it is important to try to choose $g(v)$ such that 
$g^2(v)\neq v$ and neither $v$ dominate $g(v)$ nor $g(v)$ dominate $v$. 

One way to guarantee these two properties 
is that $v$ and $g(v)$ belong to an induced cycle of length at least six.
In Figure \ref{f:6cycle} we show the case of a cycle of length six. 
If the vertices of the cycle are labeled $v_0,v_1,\ldots,v_{2k+1}$, then
by defining $g(v_i)=v_{i+2}$, for every $i=0,\ldots,2k-1$, $g(v_{2k})=v_0$ and $g(v_{2k+1})=v_1$,
we get the desired property. Indeed, in this case we 
have that $g^2(v_i)\neq v_i$, for each $i=0,\ldots,2k+1$ and since the 
cycle is induced there is no vertex in the cycle dominated by another vertex in the cycle.

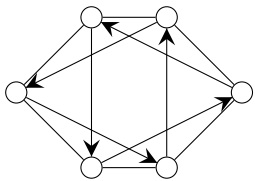
\begin{figure}[h]
\centering
\begin{tikzpicture}

\vertex[circle,  minimum size=8pt](1) at  (0,1) {};
\vertex[circle,  minimum size=8pt](2) at  (1,0) {};
\vertex[circle,  minimum size=8pt](3) at  (2,0) {};
\vertex[circle,  minimum size=8pt](4) at  (3,1) {};
\vertex[circle,  minimum size=8pt](5) at  (2,2) {};
\vertex[circle,  minimum size=8pt](6) at  (1,2) {};
\draw (1)--(2) -- (3) -- (4) -- (5) -- (6) -- (1);

\draw[decoration={markings,mark=at position 1 with
    {\arrow[scale=2,>=stealth]{>}}},postaction={decorate}
] (1) -- (3);
\draw[decoration={markings,mark=at position 1 with
    {\arrow[scale=2,>=stealth]{>}}},postaction={decorate}
] (3) -- (5);
\draw[decoration={markings,mark=at position 1 with
    {\arrow[scale=2,>=stealth]{>}}},postaction={decorate}
] (5) -- (1); 
\draw[decoration={markings,mark=at position 1 with
    {\arrow[scale=2,>=stealth]{>}}},postaction={decorate}
] (2) -- (4);
\draw[decoration={markings,mark=at position 1 with
    {\arrow[scale=2,>=stealth]{>}}},postaction={decorate}
] (4) -- (6);
\draw[decoration={markings,mark=at position 1 with
    {\arrow[scale=2,>=stealth]{>}}},postaction={decorate}
] (6) -- (2);     

\end{tikzpicture}
\caption{Defining function $g$. An arrow from $u$  to $v$ means that $v=g(u)$.}
\label{f:6cycle}
\end{figure}

Let $W$ be the set of vertices included in some induced cycle of length at least
six. If every vertex $G$ is contained in $W$,
then by applying iteratively the idea presented above, we can define $g(v)$ for each vertex $v$
such that $v$ and $g(v)$ are in an induced cycle of length at least six.
Then, we will have that $g^2(v)\neq v$ and $\ov{vg(v)}^G\neq \ov{ug(u)}^G$, because
neither $v$ dominates $g(v)$ nor $g(v)$ dominates $v$.
Therefore, for 2-connected bipartite graphs we can prove the conjecture of Zwols mentioned in the introduction since in this case every vertex belongs to $W$.

When a vertex $x$ does not belong to $W$, then for each $y\in N^2(x)$ 
there is an induced cycle of length four that contains $x$ and $y$. 
In such situation we can still define $g$ such that $g^2(w)\neq w$, for 
each $w\in V$. But, there are graphs containing a vertex $x$ such that 
for each $y\in N^2(x)$ the width of line $\overline{xy}^G$ is at least two.
In Figure \ref{f:4cycle} vertex $x$ has this property. 
We shall prove that when this happens all pair of vertices at distance two
defining the line $\overline{xy}^G$ contains $x$. 
Hence, in order to avoid $f(x)=f(z)$ we only need to define $g(z)\neq x$.
In the next lemma we prove that this can always be done since 
when three distinct vertices $x,u$ and $z$ are such that 
$d(x,u)=d(x,z)=d(u,z)=2$ and $\overline{xu}^G=\overline{xz}^G$,
then $u,z\in W$.

\begin{figure}[h]
\centering
\begin{tikzpicture}

\vertex (A)  [minimum size=12pt] {$u$};
\vertex (B) [minimum size=12pt, below right of=A] {};
\vertex (C) [minimum size=12pt, above right of=A] {};
\vertex (D) [minimum size=12pt, above right of=B] {$x$};
\vertex (E) [minimum size=12pt, above right of=D] {};
\vertex (F) [minimum size=12pt, below right of=D] {};
\vertex (G) [minimum size=12pt, above right of=F] {};
\vertex (H) [minimum size=12pt, above right of=C] {$z$};
\vertex (I) [minimum size=12pt, below right of=B] {};

\draw (A) -- (C) -- (D) -- (E) -- (G) -- (F) -- (D) -- (B) -- (A);
\draw (C) -- (H) -- (E);
\draw (B) -- (I) -- (F);

\end{tikzpicture}
\caption{The vertex $x$ always generates lines of width 2.}
\label{f:4cycle}
\end{figure}
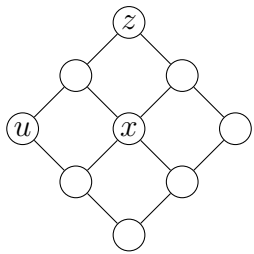

\begin{lemma}\label{l:lengthsix}
Let $u\in V$ such that $\exists x\in X$ with $u\in Y_x$
and $\exists z\in N^2_G(u)\cap  Y_x$. Then, $u,z\in W$.
\end{lemma}
\begin{proof}
 From the definition of $X$ and $Y_x$, there are $s\in X$, $u'\in Y_s$  
such that $\ov{su'}^G = \ov{ux}^G$. 
Let $v$ be a neighbor of $z$  in a shortest path between $z$ and $u'$. 
On the one hand, $v\in N_G(z)\subseteq N_G(x)$ since
$x$ dominates $z$, and then $v\in \ov{xu}^G=\ov{su'}^G$; 
on the other hand, from  part (\ref{s:ztot}) of Proposition \ref{p:uniqueline}, $d(v,s)=d(v,u')$; then $d(v,s)=d(v,u')=1$. These implies that $x=s$ and $v\in N_G(x)\cap N_G(u')\cap N_G(z)$.

Notice that the roles of $u$ and $z$ are symmetric with respect to $x$.
Hence, there is $z'$ such that $\ov{xz}^G=\ov{xz'}^G$
and there is $v'\in N_G(u)\cap N_G(z')\cap N_G(x)$.
Moreover, there are $w\in N(u)\cap N(x)\cap N(z)$ and
$w'\in N_G(u')\cap N_G(x)\cap N_G(z')$. Therefore, the 
cycle $uwzvu'w'z'v'u$ has length eight and  it is an induced cycle because $d(u,u')=d(z,z')=4$. As $u$ and $z$ belong to this cycle we get the conclusion.
\end{proof}

Now we can prove our main result for 2-connected bipartite graph without pairs of twins.

\begin{theorem}\label{t:notwins}
Let $G$ be a  2-connected bipartite graph without pairs of twins. Then
$$\ell_2(G)\geq |G|.$$
\end{theorem}
\begin{myp}{} Under the assumptions, from Corollary \ref{c:universal} we know that $G$ has no universal pairs at distance two. Moreover, from Lemma \ref{l:whenaligned} we also know that 
if there are $x,y,s,t$ such that $d(x,y)=d(s,t)=2$ and $\ov{xy}^G=\ov{st}^G$, 
then we can assume that $[yxst]$.

We prove that there exists a function $g:V\to V$ 
satisfying $d(u,g(u))=2$ for each $u\in V$, and such that 
the function $f:V \to {\cal L}_2^G$ defined by $f(u)=\ov{ug(u)}^G$ is 
injective. By Corollary \ref{c:universal} the function $f$ ranges over non-universal lines since $G$ has no pairs of twins.


The definition of $g$ is made in several steps: 
\begin{itemize}
 \item We first define $g$ in  the set $W$. 
Iteratively, we take any induced cycle $C$ of length at least six having vertices 
where $g$ has not been defined. We define $g$ in all the vertices
of the cycle. If for some of them $g$ has been previously defined, we redefine $g$ for these vertices. Let $C$ be a cycle  given by $u_0,u_1,\ldots,u_{2k+1}$, with $k\geq 2$. Then

$$ g(u_{2k})=v_0, g(u_{2k+1})=v_1, \text{ and } g(u_i) = v_{i+2}, i \in \{0,1,\ldots, 2k-1\}. $$


We have that $w(\ov{u_ig(u_i)})=1$ since neither $u_i$ dominates $u_{i+1}$, nor $u_i$ is dominated by $u_{i+1}$ as $C$ is an induced cycle of size greater than 4. 
It is clear that $g^2(u_i)\neq u_i$, for each $i=0,\ldots,2k+1$. 

\end{itemize}

To ease the presentation let $Z:=\bigcup_{x\in X}N^2_G(x)$ be the set of all the neighbors at distance 2 of vertices in the set $X$.
Notice that $Y\subseteq Z$ and that from Proposition \ref{p:uniqueline} the set $Y^2:=\bigcup_{y\in Y_x}N^2_G(y)$ is included in $Z$. We also define the set $C:= (Z\cup X\cup W)^c$.
\begin{itemize}
\item We define $g$ in $C$. Let  $u \in C$: 


If there exists $v \in N^2_G(u)$ with $g(v)\neq u$, then we define $g(u)=v$. Since $u \notin Z$ then $v \notin X$ and $w(f(u)) = 1$. 

If $g(v)=u$ for each $v \in N^2_G(u)$, then we claim that $N^2_G(u)$ 
has at least two vertices $z,z'$ such that $N_G(u)\cap N_G(z)\cap N_G(z')$ is not empty. In effect, the vertex $u \notin W$ which implies that $u$ is contained in a cycle of size four. Let $z$ be the vertex at distance 2 of $u$ in this cycle. Since the other two vertices of the cycle do not form a pair of twins, there exists a vertex $z'$ which is neighbor of only one of them and such that $d(z',u)=2$.

By using $z$ and $z'$ we define $g(u)=z$ and redefine $g(z)=z'$ (see Figure \ref{f:reg}); since $N^2(u) \cap X = \emptyset$ we have that $w(\ov{zz'}^G)=1$ and as $u\notin X$, we get $w(\ov{ug(u)}^G)=1$. Moreover, $g^2(u)=z'\neq u$ and $g^2(z')=z\neq z'$. 
Notice that with these definitions $g(C) \subseteq W \cup C$.


\begin{figure}[h]
\begin{minipage}[c]{0.5\textwidth}
\centering
\begin{tikzpicture}[scale=0.9]

\vertex[circle,  minimum size=12pt](1) {};
\vertex[circle,  minimum size=12pt](2) [above right of=1] {};
\vertex[circle,  minimum size=12pt](3) [below right of=2] {};
\vertex[circle,  minimum size=12pt](4) [above right of=3] {};
\vertex[circle,  minimum size=12pt](5) [below right of=4] {};
\vertex[circle,  minimum size=12pt](6) [above right of=5] {};
\vertex[circle,  minimum size=12pt](7) [above of=4] {$u$};
\draw (1)--(2) -- (3) -- (4) -- (5) -- (6) -- (7) -- (2);
\draw (4) -- (7);

\draw[dashed,->] (1) to [out=90,in=150] (7);
\draw[dashed,->] (3) -- (7);
\draw[dashed,->] (5) -- (7); 

\end{tikzpicture}
\end{minipage}%
\begin{minipage}[c]{0.5\textwidth}
\centering
\begin{tikzpicture}[scale=0.9]

\vertex[circle,  minimum size=12pt](1) {};
\vertex[circle,  minimum size=12pt](2) [above right of=1] {};
\vertex[circle,  minimum size=12pt](3) [below right of=2] {$z'$};
\vertex[circle,  minimum size=12pt](4) [above right of=3] {};
\vertex[circle,  minimum size=12pt](5) [below right of=4] {$z$};
\vertex[circle,  minimum size=12pt](6) [above right of=5] {};
\vertex[circle,  minimum size=12pt](7) [above of=4] {$u$};
\draw (1)--(2) -- (3) -- (4) -- (5) -- (6) -- (7) -- (2);
\draw (4) -- (7);

\draw[dashed,->] (1) to [out=90,in=150] (7);
\draw[dashed,->] (7) -- (3);
\draw[dashed,->] (5) -- (7); 
\draw[dashed,->] (3) -- (5);
\end{tikzpicture}
\end{minipage}
\caption{Redefining $g$}
\label{f:reg}
\end{figure}
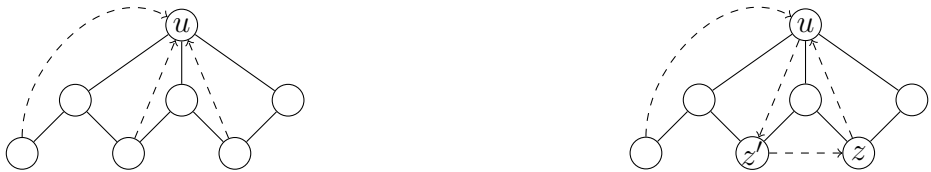

\item Now we define $g$ for  $u\in Z\setminus (X\cup W)$ such that there exists $x\in X \cap N^2_G(u)$ with $u \notin Y_x$. In this case, we set $g(u)=x$.  We have $w(f(u))=1$ because $u \notin Y_x$. If $x \in W$ then $g(x) \neq u$ because $u \notin W$. Otherwise $g(x)$ has not been defined yet, so we will show later that $g(x) \neq u$.

\item Now we define $g$ for $u\in Z\setminus (X\cup W)$ such that for all $x\in X\cap N^2_G(u)$ we have that $u \in Y_x$. In this case,  we claim there exists a vertex $z \in N^2(u)$ such that $z \notin X$ or $u \notin Y_z$. In effect, if for every $z\in N^2_G(u)$ we have that $u\in Y_z$, then,
by Proposition \ref{p:uniqueline}, every $z\in N^2_G(u)$ dominates $u$, which implies that there is a pair of twins whose common neighborhood is $N^2_G(u)\cup \{u\}$ inside the neighbors of $u$.
Thus, there is $z\in N^2_G(u)$ such that $z \notin X$ or $u\notin Y_z$; so we define $g(u) = z$. 
Clearly, $w(f(u)) = 1$, by definition. If $z \in X$, then $g(z)$ have not been defined yet. Otherwise, $z \in Z$ and $z \notin Y_x$ for all $x \in X \cup N^2(u)$, by Lemma \ref{l:lengthsix} since $u \notin W$. In this case $g(z)$ was defined in the previous step and satisfies $g(z) \in  X$.
\item The last step is to define $g$ for $u\in X\setminus W$. We pick $y\in Y_u$ arbitrarily and define $g(u)=y$. From the definition of $X$ we conclude that $w(f(u)) > 1$. Notice that $g(y)$ was already defined in previous steps; moreover, in the previous steps we always have defined $g(v)$ such that $w(f(v))=1$; so we have that $g^2(v) \neq v$ for all $v \in Z$ such that $g(v) \in X$. 
\end{itemize}

Finally we prove the injectivity of $f$. Suppose there exist $u, u' \in X$ such that $f(u) = f(u')$. Since $g^2(u) \neq u$ and $g^2(v) \neq v$, we have that $\{u, g(u)\}$ and $\{v, g(v)\}$ generate the same line. 
Since $G$ is 2-connected and has no pairs of twins, we can assume that $[g(u)uu'g(u)]$ holds. From part (\ref{s:cycle}) of Proposition \ref{p:uniqueline}  we have that $u$ and $u'$ are contained in a cycle of size $2(d(u,u') +2)$, but $u \notin W$, which implies $d(u,u')=0$ from where we obtain the injectivity of $f$. 

\end{myp}

Now we analyze the remaining cases:
\begin{theorem}\label{t:mainbip}
Let $G$ be a connected bipartite graph with at least 3 vertices. If $G\notin \{C_4, K_{2,3}\}$ then 
$$\ell_2(G)+\textsc{br}(G)\geq n.$$
\end{theorem}

\begin{myp}{} 
We proceed by induction on $n:=|G|$. If $n = 3$, then $G$ is a path with 3 vertices and satisfies $\ell_2(G) = 1$ and $\textsc{br}(G) = 2$.

Suppose that $G$ has a pendant edge $ab$ with $b$ a vertex of degree 1. Let $G':= G -b$; if $G'$ is not isomorphic with $C_4$ or $K_{2,3}$, then by induction hypothesis we have that 
$$ \ell_2(G)+\textsc{br}(G) \geq \ell_2(G') + \textsc{br}(G') + 1 \geq n. $$ 
When $G'$ is isomorphic with $C_4$ or $K_{2,3}$ a cases analysis shows that 
the graph $G$ satisfies $\ell(G)+1\geq |G|$. 
Hence, we can assume in the sequel that $G$ has no pendant edges.
If $G$ has a bridge $ab$, let $G_a$ and $G_b$ the connected components of $G - ab$ that contain $a$ and $b$, respectively. As $G$ has no pendant edge, both $G_a$ and $G_b$ 
have at least two vertices one of them of degree at least two; hence they have at least 3 vertices.
Let $G'_a$ be the subgraph of $G$ induced by $V(G_a) \cup \{b\}$ and $G'_b$ the subgraph of $G$ induced by $V(G_b) \cup \{a\}$. As they have a pendant edge they are neither $C_4$ nor $K_{2,3}$. 
For two distinct vertices $u$ and $v$ in $G'_a$  we have that 
$\ov{uv}^G \in \{ \ov{uv}^{G'_a}, \ov{uv}^{G'_a} \cup G_b\}$. 
The analogous property holds for vertices in $G'_b$. Then, it follows that 
$$ \ell_2(G) \geq \ell_2(G'_a) + \ell_2(G'_b) -1. $$
On the other hand $G'_a$ and $G'_b$ share a bridge, hence
$$ \textsc{br}(G) = \textsc{br}(G'_a) + \textsc{br}(G'_b) -1.$$
By plugging these two inequalities and using the induction hypothesis we obtain:
\begin{eqnarray*}
\ell_2(G)+\textsc{br}(G) & \geq &\ell_2(G'_a)+\textsc{br}(G'_a)+\ell_2(G'_b) + \textsc{br}(G'_b) - 2\\
&\geq & |G'_a| + |G'_b| - 2 \\
& = & n + 2 -2 = n. \\
\end{eqnarray*}
Hence, in what follows we can assume that the graph $G$ is bridgeless.
We now consider that $G$ is bridgeless and has a cut vertex $v$. 
Let $G_1$ and $G_2$ be subgraphs of $G$ such that $E(G)=E(G_1)\cup E(G_2)$ and
$V(G_1)\cap V(G_2)=\{v\}$. Then from Lemma \ref{lem:cutvertex}
 we have that 
 \begin{equation*}
\ell_2(G)\geq \ell_2(G_1)+\ell_2(G_2)+3.\\
\end{equation*}
By induction hypothesis this quantity is greater than $|G|$ unless $G_1 = G_2 = C_4$ because $\ell_2(C_4) = 1$ and $\ell_2(K_{2,3}) = 4$. When $G_1 = G_2 = C_4$ we can compute directly the value $\ell_2(G) = 7=|G|$. Hence, in the rest of the proof we can assume that $G$ is 2-connected. 

If $G$ has no pair of twins, then we obtain the conclusion from Theorem  \ref{t:notwins}.

If $G=K_{q,p}$, then from Lemma \ref{l:complete} we get 
conclusion as $\binom{p}{2}+\binom{q}{2}\geq p+q$ when $p+q\geq 6$.
For $p+q\leq 5$, $p+q=5$ implies that $G=K_{2,3}$ and $p+q=4$ implies that $G=C_4$.

To end the proof, we assume that $G$ is 2-connected, it has pairs of twins and it is not a complete bipartite graph. 

We choose $M=\{v_1, v_2\}$ as a pair of twins with $G':=G-v_1$ having as few bridges as possible.

The graph $G'\neq C_4$, as otherwise $G=K_{2,3}$, and $G'\neq K_{2,3}$, as otherwise $G=K_{3,3}$ or $G=K_{2,4}$.

From the induction hypothesis, we have that $\ell_2(G')+\textsc{br}(G')\geq |G'|=|G|-1$.

Set $\mc L'=\{\ov{xy}^G: x,y \in V(G'), d(x,y)=2\}$. 
Since $G'$ is an isometric subgraph of $G$ (i.e. for all $x,y \in V(G')$, the distance between $x$ and $y$ in $G'$ is the same as it is in $G$), we have,  for all $a,b \in   V(G')$, $\ov{ab}^G=\ov{ab}^{G'}$ or $\ov{ab}^{G'}=\ov{ab}^{G'} \cup \{v_1\}$. 
Hence 
\begin{equation}\label{eq4:1}
|\mc L'| = \ell_2(G') \ge |G|-1-\textsc{br}(G').
\end{equation}
Moreover, each line in  $\mc L'$ that contains $v_1$ must contains $v_2$.

Since $G$ is not a complete bipartite graph, there is 
$t\in G-(M\cup N(M))$. It is clear that $v_1$ is the unique vertex in $M$ which belongs to the line $\ol{v_1t}^G$. 
Hence, $\ol{v_1t}^G\notin \mc L'$ and thus, if $\textsc{br}(G')=0$, we are done by (\ref{eq4:1}). So we may assume that $G'$ has at least one bridge $ab$. 
We will prove that the choice of $M$ guarantees that there is only one bridge in $G'$.

\begin{claim}
For any bridge $ab$ of $G$, $v_2=a$ and there is a connected component that only contains $b$.
\end{claim}

\begin{proof}
Set $G_a$ be the connected component that contains $a$ and $G_b$ the one that contains $b$ in the graph $G -ab$. Without loss of generality we can assume that $v_2 \in G_a$. Since $G$ is bridgeless $v_1$ and $v_2$ must have neighbors in $G_a$ and $G_b$ which implies that $v_2 = a$. Moreover, since $G$ has no cut vertex it follows that $G_b = b$.
\end{proof}


Suppose that there exist at least two bridges in $G'$. By the claim, we know that all of them are incident with $v_2$ and $v_1$ in $G$. In particular, they form a pair of twins and the graph obtained 
if we remove one of them has no bridges, contradicting the choice of $M$. Hence, $G'$ has only one bridge.


Consider now the line $\ov{v_1v_2}^G$;  since $G$ is bipartite, $N(M)$ is an  independent set and thus $\ov{v_1v_2}^G=M \cup N(M)$. 
We claim that $\ov{v_1v_2} \notin \mc L' \cup \{\ov{v_1t}^G\}$ which gives the result by (\ref{eq4:1}). 

In effect, we first note that $\ov{v_1v_2}^G \neq \ov{v_1t}^G$, since $t\notin M\cup N(M)$. So we may assume, for the sake of contradiction, that $\ov{v_1v_2}^G \in \mc L'$. 
Let $x,y \in N(M) \cup M - \{v_1\}$ such that $\ov{xy}^G=M \cup N(M)$.  
Notice that for every pair of vertices $u,u' \in N(M) \setminus \{b\}$, the line $\ov{uu'}^G$ does not contain $b$; then we can assume that $x = b$ since $v_2$ does not have any vertex at distance 2 in $\ov{xy}^G$; but this is a contradiction because if $y \in N(M)$ then $|N(M) \setminus b| = 1$ which is not possible since $G$ has no cut vertex. 
\end{myp}

\section{Metric space with few distances}

Let $M=(V,d)$ be a metric space. Let $a\in V$ and let $V^{-a}:=V\setminus \{a\}$.
The set $V^{-a}$ endowed with the restriction of $d$ to $V^{-a}$ 
is a metric space that, in this work, we shall refer to as $M^{-a}=(V^{-a},d^{-a})$.

Notice that for a metric space $M$ defined by a graph, the metric space $M^{-a}$ may not be the same as the metric space defined for the subgraph obtained after removing vertex $a$. 

Recall that $\ell^*(M)$ denotes the number of distinct non-universal lines in $M$.
In metric spaces we have the following relation between its lines and the lines of its subspaces.

\begin{lemma}\label{l:linesind}
 For every metric space $M = (V,d)$, $\ell^*(M)\geq \ell^*(M^{-a})$,
 for any $a\in V$.
\end{lemma}
\begin{proof} 
Let $V':=V^{-a}$, $d':=d^{-a}$ and $M':=M^{-a}$.
Let $x, x',y, y'\in V'$ such that $l:= \ov{xy}^{M'} \neq l':=\ov{x'y'}^{M'}$. Since these lines are different, we can assume there exists a point $z \in V'$ such that $z \in l \setminus l'$. Since the distance between points in $V'$ does not change in $V$, it follows that $d(x,y)=|d(x,z)\pm d(z,y)|$ and $d(x',y')\neq |d(x',z)\pm d(z,y')|$ which implies that $\ov{xy}^{M} \neq \ov{x'y'}^{M}$. 

Hence, two different lines in $M'$ extend to two different lines in $M$. Therefore, $\ell^*(M)\geq \ell^*(M')$.

\end{proof}


Let us recall that $(v,v')$ is a \emph{pair of twins} of a metric space $M=(V,d)$,
if $v$ and $v'$ are two distinct points in $V$ such that $d(v,v')\neq 1$ and for all $u\notin \{v,v'\}$, $d(v,u)=d(u,v')$.

The symmetric role of vertices in a pair of twins with respect to the distance function is partially described in the following lemma.

\begin{lemma}\label{l:twinsandlines}
Let $(v,v')$ be a pair of twins on $M=(V,d)$ and let $x,y$ two distinct points in $V^{-v'}$. 
If $v\notin \{x,y\}$, then $v\in \ov{xy}^M$ if and only if $v' \in \ov{xy}^M$. 
\end{lemma}
\begin{proof}
By definition $v\in \ov{xy}^M$ if and only if $d(x,y)=|d(x,v)\pm d(v,y)|$; but as $(v,v')$ 
is a pair of twins, we can replace $v$ in previous equality by $v'$ and we get the result.
\end{proof}

To ease the presentation we denote by $\cal M^*$ 
the set of all metric spaces 
satisfying
$$\ell^*(M)+\textsc{up}(M)\geq |M|.$$

Now we prove that a metric space with at least three points which is minimal not in $\cal M^*$ cannot contain
a pair of twins $(v,v')$ such that, for every $u\in V\setminus \{v,v'\}$, $d(u,v)=1$.

\begin{proposition}\label{p:mincounternotwins}
Let $M=(V,d)$ be a minimal metric space not in $\cal M^*$ with at least three points. 
If $(v,v')$ is a pair of twins of $M$, then there is $u\in V\setminus \{v,v'\}$ such that
$d(v,u)\neq 1$. 
\end{proposition}
\begin{proof} 
For the sake of contradiction, let $M$ be a minimal metric space not in $\cal M^*$ and 
let $(v,v')$ be a pair of twins of $V$ such that for each $u\in V\setminus \{v,v'\}$, $d(v,u)=1$.
Since $d(v,v')\neq 0,1$, we have that $d(v,v')\geq 2$. As $M$ has at least three points,
there is $u\notin \{v,v'\}$. Then, for such $u$ we have $d(v,u)+d(u,v')\leq 2$, which implies that $d(v,v')=2$.
Hence, 
$$\overline{vv'}^M=\{v,v'\}\cup \{u\in V: d(u,v)=d(u,v')=1\}=V.$$ 
Thus, $(v,v')$ is a universal pair of $M$.


Let $M'=M^{-v'}$. By the minimality of $M$, the space $M'$ belongs to $\cal M^*$. Hence, 
we have that $\ell^*(M')+\textsc{up}(M')\geq |V|-1$.
From Lemma \ref{l:linesind}, we have that $\ell^*(M)\geq \ell^*(M')$.
As $M\notin \cal M^*$ we have that $|V|-1\geq \ell^*(M)+\textsc{up}(M)$.
Hence, $\textsc{up}(M') \geq \textsc{up}(M)$.

To get the contradiction we prove that $\textsc{up}(M) > \textsc{up}(M')$.
Let $(x,y)$ be a universal pair in $M'$. We prove that it is also universal in $M$.
By Lemma \ref{l:twinsandlines} this is immediate if $v\notin \{x,y\}$. 
So, we can assume that $x=v$. As we have that $d(v,y)=1=d(v',y)$ and $d(v,v')=2$, we
get that $v'\in \overline{xy}^M$, thus 
$(x,y)$ is a universal pair in $M$ as well.
To prove the strict inequality notice that 
$(v,v')$ is a universal pair in $M$ but not in $M'$.

\end{proof}

\subsection{2-metric spaces}

In this section, we prove that 2-metric spaces with at least three points belong 
to $\cal M^*$.
We first study the case when the metric space has no pairs of twins. 
In order to do that, we fix a point $v$ of the metric space and we count the different lines defined by $v$ and the other vertices of the metric space.


The following lemma summarizes the restrictions on a 2-metric space $M$ 
appearing when there are repeated lines. The first statement appears in \cite{ChCh11}.

\begin{lemma}\label{l:1}
Let $M=(V,d)$ be a 2-metric space.
Let $v,x,y,z$ points in $V$.
\begin{enumerate}[(i)]
\item\label{l:1:i} If $v,x,y$ are distinct, then $\ov{vx}^M=\ov{vy}^M$ implies $d(v,y)\neq d(v,x)$ or $(x,y)$ is a pair of twins with $d(v,x)=d(v,y)=1$.
\item\label{l:1:ii} If $d(v,y)=2$, then the only point in $\ov{vy}^M$ at distance two from $v$ is $y$.
\item\label{l:1:iii} If $d(v,y)=2$, $d(v,x)=1$, $\ov{vy}^M=\ov{vx}^M$,  
$d(v,z)=2$, $d(x,z)=1$ and $z\in \ov{xy}^M$; then $z=y$.
\end{enumerate}

\end{lemma}
\begin{proof} The first statement was proved in \cite{ChCh11}.
The second statement is direct because  if a point $u$ satisfies  $d(v,u)=d(v,y)=2$ and $u\neq y$, then $u\notin \ov{vy}^M$ by definition. For the third statement, it is immediate that $z\in \ov{vx}^M$ since $d(v,z)=2= d(v,x) + d(x,z)$. From the second statement we get that $z=y$ since $\ov{vx}^M=\ov{vy}^M$ and $d(v,y)=2$.
\end{proof}

Let $M= (V,d)$ be a 2-metric space and $v\in V$. We define the sets 
$$F=\{x\in V: d(x,v)=1\} \text{ and } S=\{y\in V: d(v,y)=2\}.$$ 

Notice that $\{\{v\},F,S\}$ is a partition of the set $V$. 
In the rest of this section, we always will consider this partition,
that is to say, $v$ is fixed for the discussion.

We consider the following sets of lines:
\begin{itemize}
\item $vF :=\{\ov{vx}^M: x\in F, \ov{vx}^M\neq V\}$.
\item $vS :=\{\ov{vy}^M: y\in S, \ov{vy}^M\neq V\}$.
\item $S^* :=\{\ov{yw}^M: y,w\in S, y\neq w\}$.
\item $FS: =\{\ov{xz}^M: x\in F, z\in S, \text{ s.t. }\exists y\in S, 
\ov{vx}^M=\ov{vy}^M, z\notin \ov{xy}^M\}$.
\end{itemize}

We give example of lines in these sets in the metric space defined by the graph of Figure \ref{f:graphmet}.
\bigskip

\begin{minipage}[c]{0.5\textwidth}

\centering
\begin{tikzpicture}[scale=0.6]
\vertex[circle,  minimum size=12pt](1) at  (0,2) {$a$};
\vertex[circle,  minimum size=12pt](2) at  (1,0) {$d$};
\vertex[circle,  minimum size=12pt](3) at  (2,2) {$b$};
\vertex[circle,  minimum size=12pt](4) at  (2,4) {$v$};
\vertex[circle,  minimum size=12pt](5) at  (3,0) {$e$};
\vertex[circle,  minimum size=12pt](6) at  (4,2) {$c$};
\draw (1)--(2) -- (3) -- (5)--(6) -- (4) -- (1);
\draw (4) -- (3) -- (6);
\draw (1) to [out=30,in=150] (6);
\draw (5,2) node{$F$};
\draw (5,0) node{$S$}; 
\end{tikzpicture}

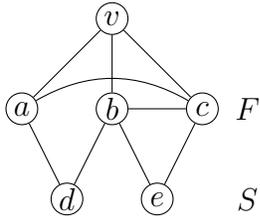
\captionof{figure}{The sets $F$ and $S$ defined by a graph.}\label{f:graphmet}
\end{minipage}%
\begin{minipage}[c]{0.5\textwidth}
\begin{itemize}
\item $\ov{va}^M = \{v,a,b,d\} \in vF$.
\item $\ov{vd}^M = \{v,a,b,d\} \in vS$.
\item $\ov{de}^M = \{ d,e,b \} \in S^*$.
\item $\ov{ad}^M = \{a,d,v,b,c \}$.
\item $\ov{ae}^M = \{a, e, c \}  \in FS$.
\end{itemize}
\end{minipage}

\bigskip
With previous notation we have the following properties
that we shall use further on.

\begin{proposition}\label{p:severalsmallresults}
Let $M=(V,d)$ be a 2-metric space. 
\begin{enumerate}[(i)]
\item \label{l:cardtrivial}  Let $\textsc{up}_F$ (resp. $\textsc{up}_S$) be the number of universal
 pairs $(v,u)$ with $u\in F$ (resp. $u\in S$) and let $\textsc{up}^{-v}$ be
 the number of universal pairs $(u,w)$ with $v\notin \{u,w\}$. 
 Then, $\textsc{up}(M)=\textsc{up}^{-v}+\textsc{up}_S+\textsc{up}_F$, 
 $|vS|=|S|-\textsc{up}_S$, and when $M$ has no pairs of 
  twins, $|vF|=|F|-\textsc{up}_F$. 
 
 \item \label{l:structFS}
$\forall \ell \in FS$, $\exists z \in S$ such that  $\ell \subseteq F \cup \{z\}$. In particular, if $FS \neq \emptyset$, then $|S| \geq 2$. 

\item \label{l:disjoint1}
$(FS \cup S^*)\cap (vS\cup vF)=\emptyset$. 
 \item \label{l:upperbound}
$|FS|+\textsc{up}^{-v}\geq |vF\cap vS|$.
\item\label{l:partialcont}
 $|vF\cup vS\cup FS|+\textsc{up}(M)\geq n-1$.
\item\label{l:disjoint2}
 $S^* \cap FS=\emptyset$.
\end{enumerate}
\end{proposition}
\begin{myp}{}
\begin{enumerate}[(i)]
\item Direct from Lemma \ref{l:1}. 

\item Let $\ell \in FS$, then there exist $x\in F, y\in S$, and $z\in S$ such that $\ell = \ov{xz}^M$, $\ov{vx}^M=\ov{vy}^M$ and  $z\notin \ov{xy}^M$.

First, notice that since $y\in \ov{vx}^M$, we get that $d(x,y)=d(v,y) - d(v,x) = 1$. 

We claim that $d(x,z) = 2$. In effect, if $d(x,z) = 1$, then $d(y,z) = 1$ because $z \notin \ov{yx}^M$; but this would imply that $z \in \ov{vx}^M$ and $z \notin \ov{vy}^M$, which is a contradiction since these lines are equal.



We have by definition that 
$$\ov{xz}^M = \{x,z\} \cup \{u: d(u,x)=d(u,z)=1\},$$ 

which implies that $v \notin \ov{xz}$. 

Now we prove that $\ov{xz}^M \cap S = \{z\}$. By contradiction, suppose there exists a point $u\in \ov{xz}^M\cap S$, with  $u\neq z$; it follows that  $d(v,u)=2$ and $d(x,u)=d(u,z)=1$ which implies that $u\in \ov{vx}^M=\ov{vy}^M$; from part (\ref{l:1:iii}) of Lemma \ref{l:1} we get that $u=y$, a contradiction since $d(z,y)= 2$. Hence,  $\ov{xz}^M\subseteq F\cup \{z\}$. 

Finally, if $FS \neq \emptyset$, there exists a point $z \in S$, which, by definition, is different from $y$, and then $|S| \geq 2$.


\item On one hand,  every line $\ell\in vF\cup vS$ contains the point  $v$ by definition; on the other hand, $v\notin \ell'$ when $\ell' \in S^*$ and from part (\ref{l:structFS}) we get that lines in $FS$ do not contain $v$.  

\item 
Let $y_1, \ldots,y_r\in S$ such that for each $i=1,\ldots,r$ 
there exists $x_i\in F$ with $\ov{vx_i}^M=\ov{vy_i}^M$ and $vF \cap vS=\{\ov{vy_i}^M:i=1,\ldots,r\}$. From part (\ref{l:1:iii}) of Lemma \ref{l:1} we deduce that if $x_i=x_j$, then $i=j$. Let $I\subseteq \{1,\ldots,r\}$ be the set of all indices such that $(x_i,y_i)$ is a universal pair of $M$. 
For each $i\notin I$ tehre exist a vertex $z_i \notin \ov{x_iy_i}^M$. We claim that $z_i \in S$. In effect, suppose that $z_i  \in F$; on one hand if $d(z_i,y_i) = d(z_i,x_i) = 1$, then $z_i \in \ov{vy_i}^M \setminus \ov{vx_i}$, a contradiction; on the other hand if $d(z_i,y_i) = d(z_i,x_i) = 2$, then $z_i \in \ov{vx_i}^M \setminus \ov{vy_i}$, a contradiction again. Hence $z_i \in S$ and $\ov{x_iz_i}^M\in FS$.

From part (\ref{l:structFS}) we get that $d(x_i,z_i) = 2$ and $\ov{x_iz_i}^M\subseteq F\cup \{z_i\}$. We shall prove that all the lines defined in  this way are different. In effect, suppose there exist $i, j$ such that $\ov{x_jz_j}^M=\ov{x_iz_i}^M$; on one hand, it follows from (\ref{l:structFS}) that $z_i=z_j$; on the other hand, 
from part (\ref{l:1:ii}) of Lemma \ref{l:1} it follows that $x_i = x_j$, since $d(x_i,z_i) = d(x_j,z_j) = 2$.

Therefore, $|FS|\geq r-|I| \geq  |vF \cap vS| - \textsc{up}^{-v}$.

\item From parts (\ref{l:disjoint1}) and (\ref{l:cardtrivial}), we get that 
\begin{eqnarray*}
|vF\cup vS\cup FS|+\textsc{up}(M) & \geq & |vF|+|vS|-|vF\cap vS|+|FS|+\textsc{up}(M) \\
&=& |F|-\textsc{up}_F+|S|-\textsc{up}_S-|vF\cap vS|+|FS|+\textsc{up}(M)\\
&\geq & |F|+|S|-|vF\cap vS|+|FS|+\textsc{up}^{-v},\\
\end{eqnarray*}
since $\textsc{up}(M)= \textsc{up}^{-v}+\textsc{up}_F+\textsc{up}_S$.

From part (\ref{l:upperbound}) we get that $|FS|+\textsc{up}^{-v}-|vF\cap vS|\geq 0$, which implies the conclusion since $|F|+|S|=n-1$.

\item Let $\ell \in FS \cap S^*$. On one hand $|\ell \cap S| \geq 2$ by definition; on the other hand from (ii) we get that  $\ell \cap S=\{z\}$ which implies that $FS \cap S^* = \emptyset$.

\end{enumerate}

\end{myp}

\begin{proposition}\label{p:notwins} A 2-metric space with no pairs of twins 
belongs to $\cal M^*$.
\end{proposition}
\begin{proof} If all distance in $M$ are 0 or 1, then every pair of points 
defines a different line and the result is immediate. Otherwise,
there is $v \in V$ such that $S$ is not empty.

From part (\ref{l:partialcont}) of Proposition \ref{p:severalsmallresults} 
we get that for such $v$,   
$$|vS\cup vF\cup FS|+\textsc{up}(M)\geq n-1.$$ 

So we only need to find a non-universal line not in $vS\cup vF\cup FS$. From
parts (\ref{l:disjoint1}) and (\ref{l:disjoint2}) of Proposition \ref{p:severalsmallresults}, we have that $S^* \cap (FS\cup vF\cup vS)=\emptyset$. 
Hence, if $ S^*$ is not empty we are done. 

Let us assume that $S^*=\emptyset$. 
Hence,  $S$ has exactly one element $y$ and the set $FS$ is empty.
Since $(v,y)$ is not a pair of twins, there is some $x\in F$ such that $d(x,y)=2$.
In particular, $v\notin \ov{xy}^M$ since $d(v,y)=d(x,y)=2$. Thus,
$\ov{xy}^M\notin vF\cup vS$ and $(x,y)$ is not 
a universal pair. 
Therefore, the line $\ov{xy}^M$ belongs to $\ell^*(M)\setminus (vF\cup vS)$
which finishes the proof.
\end{proof}

\begin{theorem}\label{t:diamtwo}
Every finite 2-metric space with at least three points belongs to $\cal M^*$.
\end{theorem}

\begin{proof} For the sake of contradiction, let $M=(V,d)$ a 2-metric space
 which is minimal not in $\cal M^*$. 

From Proposition \ref{p:notwins} we can assume that $M$ has a pair of twins $(v,v')$. 

Let $M'=M^{-v'}$ and $V^{-v'}$. 
Since $M$ is minimal not in $\cal M^*$  we have that 
\begin{equation} \label{e:min}
\ell^*(M')+\textsc{up}(M')\geq n - 1 \geq \ell^*(M)+\textsc{up}(M).
\end{equation}

Let 
$$U:=\{ u \in V': d(v,u) = 2, \ov{vu}^{M'} = V' \}$$ 
and 
$$W :=\{ u \in V': d(v,u) = 2, \ov{vu}^{M'} \neq V' \}.$$

Notice that $$V=\{v,v'\}\cup\{z: d(v,z)=d(v',z)=1\}\cup U\cup W.$$

As $M$ is minimal not in $\cal M^*$, Proposition \ref{p:mincounternotwins} implies that $|U| + |W| > 0$.

For each point $u \in U \cup W$, the line $\ov{v'u}^M$ contains $v'$ and 
does not contain $v$. Part (\ref{l:1:ii}) of Lemma \ref{l:1} implies  that all these lines are distinct and Lemma  \ref{l:twinsandlines} implies that none of these lines can be generated by two points in $V'$. Additionally, when $u \in U$, the line $\ov{vu}^{M'}$ is not counted in $\ell^*(M')$ because it is universal in $M'$.
Hence, we get 
$$\ell^*(M)\geq \ell^*(M')+ 2|U| + |W|.$$

Plugging this inequality with (\ref{e:min}) we get
\begin{equation} \label{e:muw}
 \textsc{up}(M') - \textsc{up}(M) \geq 2|U| + |W|. 
\end{equation}
 
From Lemma \ref{l:twinsandlines} we deduce that any universal pair $(x,y)$ of $M'$ with $\{x,y\}\cap \{v,v'\}=\emptyset$ is also a universal pair in $M$. Moreover, any universal pair of $M'$ which is not universal in $M$  contains $v$ and a point from the set $U$. Hence $\textsc{up}(M) + |U| = \textsc{up}(M')$.
Replacing in (\ref{e:muw}) we obtain
$$ 0 \geq  |U| + |W|,$$  

which is a contradiction.



\end{proof}


\begin{thebibliography}{10}

\bibitem{metricSpace}
{P.~Aboulker, X.~Chen, G.~Huzhang, R.~Kapadia and C.~Supko.}
\emph{Lines, Betweenness and Metric Spaces. Discrete \& Computational Geometry} 
\textbf{56}(2)  (2016), 427-448.


\bibitem{AK}
{P. Aboulker and R. Kapadia}, 
{The Chen-Chv{\'{a}}tal conjecture for metric spaces induced by distance-hereditary graphs},
\emph{Eur. J. Comb.}, \textbf{43} (2015), {1--7}.

\bibitem{AMRZ}
{P. Aboulker, M. Matamala, P. Rochet, J. Zamora} 
{A new class of graphs that satisfies the Chen-Chvátal Conjecture.} 
\emph{J. of Graph Theory.} \textbf{87}(1) (2018), 77--88.

\bibitem{BBCCCCFZ} 
{L. Beaudou, A. Bondy, X. Chen, E. Chiniforooshan, M.  Chudnovsky,
V.  Chv{\'{a}}tal, N. Fraiman and Y.  Zwols},
{A De Bruijn-Erd{\H{o}}s Theorem for Chordal Graphs},
\emph{Electr. J. Comb.}, \textbf{22}(1), {P1.70},  {2015}.

\bibitem{CC}
X. Chen and V. Chv\'{a}tal,
Problems related to a de Bruijn - Erd\H os theorem,
\emph{Discrete Applied Mathematics} \textbf{156} (2008),  2101--2108.

\bibitem{ChCh11}
E. Chiniforooshan and  V. Chv{\'{a}}tal, 
 A de Bruijn - Erd{\H{o}}s theorem and metric spaces,
 \emph{Discrete Mathematics {\&} Theoretical Computer Science},
 \textbf{13} (1) (2011), {67--74}.

\bibitem{ChvatalMetric}
V. Chv\'atal,
Sylvester-Gallai theorem and metric betweenness. 
\emph{Discrete \& Computational Geometry} \textbf{31} (2) (2004), 175--195.

\bibitem{Chvatal2}
V. Chv\'atal,
A de Bruijn-Erd\H{o}s theorem for 1-2 metric spaces,
 \emph{Czechoslovak Mathematical Journal} \textbf{64} (1) (2014), 45--51.

\bibitem{dbe}
N.~G.~De~Bruijn, P.~Erd\H{o}s,
On a combinatorial problem,
{\em Indagationes Mathematicae\/} {\bf 10} (1948),
421--423.


\bibitem{L1}
I. Kantor, B. Patk\'os,
Towards a de Bruijn-Erd\H{o}s theorem in the L1-metric,
{\em Discrete \& Computational Geometry\/} {\bf 49} (2013),
659--670.

\bibitem{YZ}
Y. Zwols, personal communication.




\end{thebibliography}
\end{document}